\documentclass[11pt]{amsart}
\usepackage[applemac]{inputenc}
\usepackage[T1]{fontenc}
\usepackage{amsthm}
\usepackage{latexsym,amsmath,amsfonts,amscd,amssymb}
\usepackage{graphicx}
\usepackage{xypic}
\usepackage{array,booktabs,arydshln}
\usepackage{setspace}
\usepackage{fancyhdr}
\usepackage{fancybox}
\usepackage{multicol}
\usepackage[all]{xy}
\newtheorem{theorem}{Theorem}
\newtheorem{lemma}[theorem]{Lemma}
\newtheorem{proposition}[theorem]{Proposition}
\newtheorem{corollary}[theorem]{Corollary}
\newtheorem{definition}[theorem]{Definition}
\newtheorem{theoremint}{Theorem}
\newtheorem*{conjecture}{Conjecture}

\theoremstyle{definition}

\newcommand{\pn}{{\mathbb{P}^n}}
\newcommand{\p}[1]{{\mathbb{P}^{#1}}}
\newcommand{\op}[1]{{\mathcal O}_{\mathbb{P}^{#1}}}

\newcommand{\KK}{{\mathbb K}}
\newcommand{\ZZ}{{\mathbb Z}}
\newcommand{\PP}{{\mathbb P}}

\newcommand{\calh}{{\mathcal H}}
\newcommand{\calu}{{\mathcal U}}

\newcommand{\cali}{{\mathcal I}}
\newcommand{\caln}{{\mathcal N}}
\newcommand{\OO}{{\mathcal O}}

\newcommand{\gothm}{\mathfrak{m}}

\newcommand{\im}{\operatorname{Im}}

\makeindex

\begin{document}

\title{$p$-Buchsbaum rank $2$ bundles on the projective space}

\author{Marcos Jardim}
\author{Simone Marchesi}

\address{IMECC - UNICAMP \\
Departamento de Matem\'atica \\ Caixa Postal 6065 \\
13083-970 Campinas-SP, Brazil}

\begin{abstract}
It has been proved by various authors that a normalized, $1$-Buchsbaum rank $2$ vector bundle on $\PP^3$ is a nullcorrelation bundle, while a normalized, $2$-Buchsbaum rank $2$ vector bundle on $\PP^3$ is an instanton bundle of charge $2$. We find that the same is not true for $3$-Buchsbaum rank $2$ vector bundles on $\PP^3$, and propose a conjecture regarding the classification of such objects.
\end{abstract}

\maketitle

\section*{Introduction}

A coherent sheaf $E$ on $\PP^3$ is said to be \emph{$p$-Buchsbaum} if $p$ is the minimal power of the irrelevant ideal which annihilates $H^1_*(E)$. The complete list of $p$-Buchsbaum rank 2 bundles on $\PP^3$ for $p\le2$ has been established by several authors, see for example \cite{E,ES,KR,MMR,MR}. More precisely, we have the following.

\begin{theoremint}\label{thm-equiv}
Let $E$ be a normalized $p$-Buchsbaum rank 2 vector bundle on $\PP^3$. Then
\begin{itemize}
\item $p=0$ if and only if $E$ is direct sum of line bundles;\\
\item $p=1$ if and only if $E$ is a null correlation bundle, i.e. an instanton bundle of charge $1$;\\
\item $p=2$ if and only if $E$ is an instanton bundle of charge $2$.
\end{itemize}
\end{theoremint}

After examining this list, two questions natually arise. First, is every rank $2$ instanton bundle of charge $k$ on $\PP^3$ $k$-Buchsbaum? Second, since every bundle is $p$-Buchsbaum for some sufficiently high $p$, for which values of $p$ can we find a $p$-Buchsbaum rank $2$ bundle which is not instanton?  

The goal of this paper is to provide partial answers to these questions. In particular, we show that every rank $2$ instanton bundle of charge $3$ is $3$-Buchsbaum. However, this is false for instantons of higher charge. On the other hand, we show that the generic instanton of charge $4$ or $5$ is also $3$-Buchsbaum. In addition, we provide an explicit example of a $3$-Buchsbaum bundle of rank $2$ which is not an instanton, and conjecture that every $3$-Buchsbaum rank $2$ bundle on $\PP^3$ is one of these.

\paragraph{\bf Acknowledgments.}
MJ is partially supported by the CNPq grant number 302477/2010-1 and the FAPESP grant number 2011/01071-3. 
SM is supported by the FAPESP post-doctoral grant number 2012/07481-1. We started thinking about instantons as $p$-Buchsbaum bundles after conversations with Joan Pons-Llopis; we thank him for several fruitful discussions.

%%%%%%%%%%%%%%%%%%%%%%%%%%%%%%%%%%%%%%%%%%%%%%%%%%%%%%%%%%%%%%%%%%%%%%%%%%%%%
%%%%%%%%%%%%%%%%%%%%%%%%%%%%%%%%%%%%%%%%%%%%%%%%%%%%%%%%%%%%%%%%%%%%%%%%%%%%%%%%%%%%%%%%%%%%%%%%%%%%%%%%%%%%%%%%%

\section{Preliminaries}\label{sec-prel}

In this section we will fix the notation and recall the basic definitions used throughout this paper.

%-----------------------------------------------------------------

\subsection{Buchsbaum sheaves}

Let $\KK$ be an algebraically closed field of characteristic zero. Let us denote by $S = \KK[x_0,x_1,x_2,x_3]$ the ring of polynomials in four variables, so that $\PP^3:={\rm Proj}(S)$, and let $\gothm = (x_0,x_1,x_2,x_3)$ denote the irrelevant ideal.

Let $V$ be a $\KK$-vector space of dimension $m+1$, with $V^*$ denoting its dual. The projective space
$\PP(V) = \PP^m$ is understood as the set of equivalence classes of $m$-dimensional subspaces of $V$, or, equivalently, the equivalence classes of the lines of $V^*$.

Given a coherent sheaf $E$ on $\PP^3$, consider the following  graded $S$-module:
$$
H^1_*(E) = \bigoplus_{n\in \ZZ} H^1(E(n)).
$$

\begin{definition}
A coherent sheaf $E$ on $\PP^3$ is said to be $p$-Buchsbaum if and only if $p$ is the minimal power of the irrelevant ideal which annihilates the $S$-module $H^1_*(E)$, i.e. 
$$ p=\min\left\{t \: | \: \gothm^t H^1_*(E) = 0 \right\}. $$
\end{definition}

In this work, we will only consider locally free sheaves on $\p3$.  

%-----------------------------------------------------------------

\subsection{Monads and regularity}

Recall that a \emph{monad} on a projective variety $X$ fo dimension $n$ is a complex of locally free sheaves on $X$ of the form
$$
M_\bullet : A \stackrel{\alpha}{\longrightarrow} B \stackrel{\beta}{\longrightarrow} C 
$$
such that the map $\alpha$ is injective and the map $\beta$ is surjective. It follows that
$E:=\ker\beta/\im\alpha$ is the only nontrivial cohomology of the complex $M_\bullet$. The coherent sheaf $E$ is called the \emph{cohomology} of $M_\bullet$; it is locally free if and only if the map $\alpha$ is injective in every fiber.

The monad $M_\bullet$ is called a \emph{Horrocks monad} if, in addition:
\begin{description}
\item[i)] $A$ and $C$ are direct sum of invertible sheaves,
\item[ii)] $H^1_*(B) = H^{n-1}_*(B) = 0$.
\end{description}
Furthermore, the monad is also called \emph{minimal} if it satisfies
\begin{description}
\item[iii)] no direct sum of $A$ is isomorphic to a direct sum of $B$,
\item[iv)] no direct sum of $C$ is the image of a line subbundle of $B$.
\end{description}

Let us recall the following result on minimal Horrocks monads, cf. \cite[Thm 2.3]{JM}.

\begin{theorem}\label{thm-JM}
Let $X$ be an arithmetically Cohen--Macaulay variety of dimension $n\geq 3$, and let $E$ be a locally free sheaf on $X$. Then there is a 1-1 correspondence between collections
$$
\{n_1,\ldots,n_r,m_1,\ldots,m_s\} \: \mbox{with}\: n_i \in H^1(E^\lor \otimes \omega_X(k_i)) \: \mbox{and}\: m_j \in H^1(E(-l_j))
$$
for integers $k_i$'s and $l_j$'s, and equivalence classes of Horrocks monads of the form
$$
M_\bullet : \bigoplus_{i=1}^r \omega_X(k_i) \stackrel{\alpha}{\longrightarrow} F \stackrel{\beta}{\longrightarrow} \bigoplus_{j=1}^s \OO_X(l_j).
$$
whose cohomology is isomorphic to $E$.

Moreover, the correspondence is such that $M_\bullet$ is minimal if and only if the elements $m_j$ generate
$H^1_*(E)$ and the elements $n_i$ generate $H^1_*(E^\lor \otimes \omega_X)$ as modules.
\end{theorem}

Recall that a coherent sheaf $E$ on $\PP^n$ is said to be $m$-regular in the sense of Castelnuovo--Mumford if $H^i(\PP^n, E(m-i)) = 0$ for $i>0$. Costa and Mir\'o-Roig studied in \cite{CMR} the Castelnuovo--Mumford regularity of the cohomology of a certain class of monads which include monads of the following form:
\begin{equation}\label{model}
\OO_{\PP^{3}}(-l)^{\oplus k} \stackrel{\alpha}{\longrightarrow} \bigoplus_{j=1}^{2+2k}\OO_{\PP^{3}}(b_j) 
\stackrel{\beta}{\longrightarrow} \OO_{\PP^{3}}(d)^{\oplus k} ,
\end{equation}
where $l,k,c\ge1$ and $-l<b_1\le\dots\le b_{2+2k}<d$. Especializing \cite[Thm 3.2]{CMR} to monads of the form (\ref{model}), one obtains the following result.

\begin{proposition} \label{prop-CMR}
If $E$ is the cohomology of a monad of the form (\ref{model}), then $E$ is $m$-regular for any integer $m$ such that
$$ m \ge {\rm max} \{ (k+2)d-(b_1+\cdots+b_{k+3})-2, l \} . $$
\end{proposition}

%-----------------------------------------------------------------

\subsection{Cohomology of generic instanton bundles}

Recall that a bundle $E$ of rank $2$ on $\PP^{3}$ is called an \emph{instanton bundle} if it is isomorphic to the cohomology of a monad of the following form:
\begin{equation}\label{def-inst}
\OO_{\PP^{3}}(-1)^{\oplus k} \stackrel{\alpha}{\longrightarrow} \OO_{\PP^{3}}^{\oplus 2+2k}
\stackrel{\beta}{\longrightarrow} \OO_{\PP^{3}}(1)^{\oplus k}
\end{equation}
The integer $k$ is called the \emph{charge} of $E$; notice that $c_1(E)=0$ and $c_2(E)=k$. Note also that nullcorrelation bundles are precisely instanton bundles of charge $1$.

Alternatively, an instanton bundle can also be defined as a bundle $E$ on $\PP^{3}$ with $c_1(E)=0$ and satisfying the following cohomological conditions:
$$ h^0(E(-1)) = h^1(E(-2)) = h^{2}(E(-2)) = h^{3}(E(-3)) = 0 . $$

The Hilbert polynomial of an instanton bundle is given by
\begin{eqnarray}\label{hilb poly}
P_E(t) & = & k \chi(\op3(t-1)) + k \chi(\op3(t+1)) - 2(k+1)\chi(\op3(t))  =  \\
 & = & \frac{1}{3} (t+2)((t+3)(t+1)-3k) = \frac{1}{3} (t+2)(t+2+\sqrt{3k+1})(t+2-\sqrt{3k+1}).
\end{eqnarray}
Note also that $P_E(t)=h^0(E(t))-h^1(E(t))$ for $t\ge-2$. 

On another direction, recall that a coherent sheaf $F$ on $\PP^{3}$ is said to have \emph{natural cohomology} if  for each $t\in\ZZ$, at most one of the cohomology groups $H^p(F(t))$, where $p=0,\dots,3$, is nonzero; every torsion free coherent sheaf with natural cohomology is in fact locally free \cite[Lemma 1.1]{HH1}. In addition, every rank $2$ locally free sheaf with $c_1=0$, $c_2>0$ and natural cohomology is an instanton bundle \cite[p. 365]{HH1}.

Hartshorne and Hirschowitz have shown in \cite{HH1} that the generic instanton bundle has natural cohomology. More precisely, let $\cali(k)$ denote the moduli space of rank $2$ locally free instanton sheaves of charge $k$; this is known to be an affine \cite{CO}, irreducible \cite{T1,T2}, nonsingular variety of dimension $8k-3$ \cite{JV}. Let $\caln(k)$ denote the subset of $\cali(k)$ consisting of instanton bundles with natural cohomology; it is easy to see that $\caln(k)$ is open within $\cali(k)$, and \cite[Thm 0.1 (a)]{HH1} tells us that his is nonempty.

More recently, Eisenbud and Schreyer have introduced the notion of \emph{supernatural bundles}, see \cite[p. 862]{ES}: a locally free sheaf on $\p3$ is called supernatural if it has natural cohomology and its Hilbert polynomial has distinct integral roots. Therefore we see that there exist a rank $2$ supernatural bundle with $c_1=0$ and $c_2=k>0$ if and only if $3k+1$ is a perfect square; the first three possible values for $k$ are $k=1,5,8$.

\section{Instanton vs Buchsbaum}\label{sec-inst}

We start by introducing the following function on the positive integers
$$ m(k) = \left\lfloor\sqrt{3k+1}-2 \right\rfloor, $$
where $\left\lfloor\cdot\right\rfloor$ denotes the largest positive integer which is smaller than or equal to the argument.

\begin{proposition}\label{prop p}
A rank $2$ instanton bundle $E$ is $p$-Buchsbaum if and only if $h^1(E(p-2))\ne0$ and $h^1(E(p-1))=0$. In addition, every rank $2$ instanton bundle of charge $k$ is $p$-Buchsbaum for some $m(k)+2\le p\le k$.
\end{proposition}
\begin{proof}
By Theorem \ref{thm-JM}, we get that $H^1_*(E)$ is generated in $H^1(E(-1))$. Thus if $h^1(E(p-2))\ne0$ and $h^1(E(p-1))=0$ (and hence $h^1(E(t))=0$ for every $t\ge p-1$), then $H^1_*(E)$ must be $p$-Buchsbaum. Conversely, if $E$ is $p$-Buchsbaum, then $h^1(E(p-2))\ne0$ (otherwise, $H^1_*(E)$ would be annihilated by the $(p-1)$-th power of the irrelevant ideal) and $h^1(E(p-1))=0$.

By Proposition \ref{prop-CMR}, we have that $E$ is $k$-regular (cf. also \cite[Cor. 3.3]{CMR}). Hence $H^1(E(k-1)) = 0$, and it follows that every rank $2$ instanton bundle is at most $k$-Buchsbaum.

Finally, note from (\ref{hilb poly}) that for $-1\le t\le m(k)$ we have $\chi(E(t))<0$. Since $h^3(E(t))=0$ in this range, it follows that $h^1(E(t))\ne 0$ for $t=m(k)$. Thus every rank $2$ instanton bundle is at least $(m(k)+2)$-Buchsbaum.
\end{proof}

Since $m(3)+2=3$, the first immediate consequence of the previous Proposition is given by the following Corollary.

\begin{corollary}
Every rank $2$ instanton bundle of charge $3$ is $3$-Buchsbaum.
\end{corollary}

However, it is not true that every rank $2$ instanton bundle of charge $3$ has natural cohomology, as observed in 
\cite[Example 1.6.1]{HH1}. Indeed, recall that an instanton bundle $E$ is called a \emph{'t Hooft instanton} if
$h^0(E(1))\ne 0$, cf. \cite{BF}; more formally, consider the set
$$ \calh(k) := \{ E \in \cali(k) ~|~ h^0(E(1))\ne 0 \} ~, $$
which is known to be a locally closed subvariety of $\cali(k)$ of dimension $5k+4$, irreducible and rational \cite[Thm. 2.5]{BF}. On the other hand, let $\calu(k):=\cali(k)\setminus\caln(k)$, the subvariety of ``unnatural" instanton bundles.

\begin{lemma}
For every $k\ge3$, we have $\calh(k)\subset\calu(k)$, while $\calh(3)=\calu(3)$.
\end{lemma}

\begin{proof}
If $E$ is a rank $2$ instanton bundle of charge $k\ge3$, then $h^1(E(1))\ne0$ (because $\chi(E(-1))<0$). Hence if $E$ is a 't Hooft instanton, then it does not have natural cohomology, showing that $\calh(k)\subset\calu(k)$.

Conversely, let now $E$ be a rank $2$ instanton bundle of charge $3$ which does not have natural cohomology. We then know that
\begin{itemize}
\item[(i)] $h^0(E(t))=0$ for $t\le0$;
\item[(ii)] $h^1(E(t))=0$ for $t\ne-1,0,1$;
\item[(iii)] $h^2(E(t))=0$ for $t\ne-5,-4,-3$;
\item[(iv)] $h^3(E(t))=0$ for $t\ge-4$.
\end{itemize}
The last two claims are obtained by Serre duality and the fact $E\simeq E^*$. Therefore the only way in which $E$ may fail to have natural cohomology is if $h^0(E(1))=h^3(E(-5))\ne0$. It follows that $\calu(3)\subset\calh(3)$.
\end{proof}

It would be interesting to determine properties of the $\calu(k)$ for $k\ge4$, particularly its dimension and number of irreducible components. The previous lemma tells us that $\dim\calu(k)\ge 5k+4$.

Another immediate consequence of Proposition \ref{prop p} is the following.

\begin{corollary}
The generic rank $2$ instanton bundle of charge $k$ is $(m(k)+2)$-Buchsbaum.
\end{corollary}

In particular, since $m(4)+2=m(5)+2=3$, the generic rank $2$ instanton bundle of charges $4$ and $5$ are $3$-Buchsbaum, while instanton bundles of charge $k\ge6$ are at least $4$-Buchsbaum.

%%%%%%%%%%%%%%%%%%%%%%%%%%%%%%%%%%%%%%%%%%%%%%%%%%%%%%%%%%%%%%%%%%%%%%%%%%%%%%%%%%%%%%%%%%%%%%%%%%%%%%%

\section{A $3$-Buchsbaum rank $2$ bundle with $c_1=-1$}\label{sec-count}

Theorem \ref{thm-equiv} tells us, in particular, that the first Chern class of every $1$- and $2$-Buchsbaum rank $2$ bundle on $\PP^{3}$ is zero. In this section, we show that the same is not true for $p$-Buchsbaum bundles with $p\ge 3$, providing an example of a $3$-Buchsbaum rank $2$ bundle with $c_1=-1$.

Indeed, consider the monad
\begin{equation}\label{c1=-1}
\OO_{\PP^3}(-2) \stackrel{\alpha}{\longrightarrow} \OO_{\PP^3}^{\oplus 2} \oplus \OO_{\PP^3}(-1)^{\oplus 2}
\stackrel{\beta}{\longrightarrow} \OO_{\PP^3}(1) ~~,
\end{equation}
which is the simplest example of a class of monads originally introduced by Ein in \cite[eq. 3.1.A]{Ein}. The existence of such monads can be easily established; consider for instance the following explicit maps
$$ \alpha = \left( \begin{array}{c} -z^2 \\ -w^2 \\ x \\ y \end{array} \right) ~~{\rm and}~~
\beta = \left(\begin{array}{cccc} ~x~ \\ ~y~ \\ ~z^2~ \\ ~w^2~ \end{array} \right) $$
where $[x:y:z:w]$ are homogeneous coordinates on $\PP^3$. 

Let $F$ denote the locally free cohomology of a monad of the form (\ref{c1=-1}); it is a rank $2$ bundle with $c_1(F)=-1$ and $c_2(F)=2$. Ein claims in \cite[p. 21]{Ein}, without proof, that $F$ is stable. For the sake of completeness, we include a proof below.

\begin{lemma}
Every locally free sheaf $F$ obtained as the cohomology of a monad of the form (\ref{c1=-1}) is stable.
\end{lemma}  

\begin{proof}
First consider the kernel bundle $K:=\ker\beta$ defined by the sequence
$$ 0 \to K \to \op3^{\oplus 2} \oplus \op3(-1)^{\oplus 2} \stackrel{\beta}{\longrightarrow} \op3(1). $$
It follows from \cite[Thm. 2.7]{BS} that $K$ is semistable (but not stable). Therefore, since  $\mu(K)=-1$, we must have $h^0(K)=0$. Now, from the sequence
$$ 0 \to \op3(-2) \stackrel{\alpha}{\longrightarrow} K \to F \to 0 $$
we have that $h^0(F)=0$, which implies that $F$ is stable.
\end{proof}

We now show that the bundles considered in this Section are $3$-Buchsbaum.

\begin{proposition}
Every locally free sheaf $F$ obtained as cohomology of a monad of the form (\ref{c1=-1}) is $3$-Buchsbaum.
\end{proposition}

\begin{proof}
By Theorem \ref{thm-JM}, we get that $H^1_*(F)$ is generated in $H^1(F(-1))$. On the other hand, Proposition \ref{prop-CMR} tells us that $F$ is $3$-regular, thus $h^1(F(2))=0$. 

If we also had $h^1(F(1))=0$, $F$ would be $2$-Buchsbaum, which, by Theorem \ref{thm-equiv} cannot happen. Therefore $F$ must be $3$-Buchsbaum.
\end{proof}

Note also that, since $h^0(F(1))=h^1(F(1))=1$ \cite[2.2]{HS}, such bundles do not have natural cohomology.

Based on the evidence here presented and also motivated by results due to Roggero and Valabrega in \cite{RV}, especially Propositions 5 and 6 and Theorem 2 there, we propose the following classification of $3$-Buchsbaum rank $2$ bundles on $\PP^{3}$.

\begin{conjecture}
Every normalized, $3$-Buchsbaum rank $2$ bundle on $\PP^{3}$ is either an instanton bundle of charge $3$, $4$ or $5$, if $c_1=0$, or the cohomology of a monad of the form (\ref{c1=-1}), if $c_1=-1$.
\end{conjecture}

\bigskip

Finally, let us comment on $p$-Buchsbaum rank $2$ bundles on $\p3$ for $p\ge4$. An interesting, possible source of examples of such bundles is provided by Ein's \emph{generalized nullcorrelation bundles}, described in \cite{Ein}. These are bundles obtained as cohomologies of monads of the following two types:

\begin{equation}\label{gncb}
\OO_{\PP^3}(-d) \longrightarrow \OO_{\PP^3}(-b) \oplus \OO_{\PP^3}(-a) \oplus \OO_{\PP^3}(a) \oplus \OO_{\PP^3}(b) \longrightarrow \OO_{\PP^3}(d) ~~,
\end{equation}
and
\begin{equation}\label{gncbmenus}
\OO_{\PP^3}(-d-1) \longrightarrow \OO_{\PP^3}(-b-1) \oplus \OO_{\PP^3}(-a-1) \oplus \OO_{\PP^3}(a) \oplus \OO_{\PP^3}(b) \longrightarrow \OO_{\PP^3}(d) ~~,
\end{equation}
where $d > b \geq a \geq 0$. Let us denote the cohomology of such monads by $E_{a,b,d}$ and $F_{a,b,d}$, respectively.

Note that, by Theorem \ref{thm-JM} and Proposition \ref{prop-CMR}, $H^1_*(E_{a,b,d})$ is generated in degree $-d$, and that $E_{a,b,d}$ is $(3d-2)$-regular when $d\ge1$. Therefore, such bundles are at most $(4d-3)$-Buchsbaum, being precisely $(4d-3)$-Buchsbaum provided $h^1(E_{a,b,d}(3d-4))\ne0$.

Similarly, note that $H^1_*(F_{a,b,d})$ is generated in degree $-d$, and that $F_{a,b,d}$ is $3d$-regular. Therefore, such bundles are at most $(4d-1)$-Buchsbaum, being precisely $(4d-1)$-Buchsbaum provided $h^1(F_{a,b,d}(3d-2))\ne0$.

%\nocite{*}

%\bibliographystyle{alpha}
%\bibliography{buchsbaum}

\end{document}